\documentclass[11pt]{article}
\usepackage[letterpaper,margin=1.1in]{geometry}
\usepackage[T1]{fontenc}
\usepackage{lmodern}
\usepackage{amsmath,amssymb,amsthm,mathtools}
\usepackage{microtype}
\usepackage{xcolor}
\usepackage{graphicx}
\usepackage{titlesec}
\usepackage[colorlinks=true,linkcolor=blue,citecolor=blue,urlcolor=blue]{hyperref}
\hypersetup{pdftitle={Rational curves on toroidal compactifications of Picard modular varieties},pdfauthor={Soheil Memariansorkhabi and Sai-Kee Yeung}}
\hypersetup{bookmarksopen=true,bookmarksnumbered=true,pdfpagemode=UseOutlines}

\titleformat{\section}
  {\normalfont\large\bfseries\centering}
  {\thesection.}{0.35em}{}
\newcommand{\address}[1]{#1}
\newcommand{\email}[1]{{\hypersetup{urlcolor=black}\href{mailto:#1}{\texttt{#1}}}}

\newtheorem{maintheorem}{Theorem}

\newtheorem{theorem}{Theorem}[section]
\newtheorem{lemma}[theorem]{Lemma}
\newtheorem{proposition}[theorem]{Proposition}
\theoremstyle{remark}
\newtheorem{remark}[theorem]{Remark}

\newcommand{\B}{\mathbb B}
\newcommand{\C}{\mathbb C}
\newcommand{\Q}{\mathbb Q}
\newcommand{\R}{\mathbb R}
\newcommand{\Z}{\mathbb Z}
\newcommand{\PP}{\mathbb P}
\newcommand{\HH}{\mathbb H}
\newcommand{\OO}{\mathcal O}
\DeclareMathOperator{\SU}{SU}
\DeclareMathOperator{\U}{U}
\DeclareMathOperator{\SL}{SL}
\DeclareMathOperator{\GL}{GL}

\DeclareMathOperator{\PU}{PU}
\DeclareMathOperator{\diag}{diag}
\DeclareMathOperator{\Mat}{Mat}

\DeclareMathOperator{\Ric}{Ric}
\DeclareMathOperator{\Span}{Span}
\newcommand{\Gmod}{\Gamma_{\SL_2}(3)}
\newcommand{\Pset}{\mathcal P}
\numberwithin{equation}{section}
\title{Rational curves on toroidal compactifications\texorpdfstring{\\}{ }of Picard modular varieties}
\author{Soheil Memariansorkhabi \and Sai-Kee Yeung}
\date{}

\begin{document}
\maketitle

\begin{abstract}
It was known that the canonical bundle of a smooth toroidal compactification of a noncompact finite-volume ball quotient is nef in dimension $n\geq3$ and ample in dimension $n\geq6$. We prove that the canonical bundle is ample in dimension $n\geq4$ and show that this dimension bound is sharp. We construct infinitely many pairwise noncommensurable noncompact Picard modular threefolds whose smooth projective toroidal compactifications have non-ample canonical bundle. Each compactification contains a smooth rational curve meeting the interior and having canonical degree zero.

More generally, in every dimension $n\geq2$, we construct infinitely many pairwise noncommensurable noncompact Picard modular varieties whose smooth toroidal compactifications contain smooth rational curves meeting the interior. In particular, such curves persist in arbitrarily high dimension, even when the canonical bundle is ample.

\end{abstract}

\section*{Introduction}
\phantomsection
\addcontentsline{toc}{section}{Introduction}

Let $X=\Gamma\backslash\B^n$ be a noncompact finite-volume ball quotient of dimension $n\geq2$ with $\Gamma$ neat, and let $\overline{X}$ be its smooth projective toroidal compactification, with boundary $D=\overline{X}\setminus X$. This compactification is uniquely determined by $X$ \cite[Proposition~2.3]{DiCerboDiCerbo}.

In dimension $n=2$, Hirzebruch constructed a ball quotient whose smooth toroidal compactification is the blowup of an abelian surface at one point \cite{Hirzebruch}. The exceptional divisor is a rational curve not contained in the boundary. His construction yields infinitely many distinct ball quotient surfaces with birational smooth toroidal compactifications \cite{DiCerboStoverMultiple}. Di Cerbo and Stover later classified the five minimum-volume ball quotient surfaces admitting smooth toroidal compactifications \cite{DiCerboStover}; see also \cite[\S 3.6]{PrasadYeung} for earlier work on classification. From \cite{DiCerboStover}, one is this Hirzebruch example, while the other four are blowups of bielliptic surfaces at one point. Therefore each of the five compactifications contains an exceptional rational curve not contained in the boundary. Since the exceptional divisor $E$ of a blowup at a point satisfies $K_{\overline{X}}\cdot E=-1$, the canonical bundle is not nef and hence is not ample in any of these five cases.

The situation changes in dimension $n\geq3$.
Di Cerbo and Di Cerbo \cite[Theorem~1.1]{DiCerboDiCerbo}
proved that $K_{\overline{X}}$ is nef for $n\geq3$.
Bakker and Tsimerman
\cite[Theorem~C and Corollary~B]{BakkerTsimerman}
proved that $K_{\overline{X}}$ is big for $n\geq4$ and
ample for $n\geq6$. We show that it is ample for $n\geq 4:$

\begin{maintheorem}\label{thm:ample canonical}
    Let $X=\mathbb B^n/\Gamma$ be a finite-volume ball quotient of
dimension $n\geq4$, with neat $\Gamma.$
Then the canonical bundle of its smooth projective toroidal
compactification $\overline X$ is ample.
\end{maintheorem}

Moreover, we show that this dimension threshold is sharp. We construct infinitely many Picard modular threefolds whose smooth toroidal compactifications have non-ample canonical bundle. These threefolds are arithmetic ball quotients associated with Hermitian forms over imaginary quadratic fields: 

\begin{maintheorem}\label{thm:threefolds}
There exist infinitely many pairwise noncommensurable noncompact Picard modular threefolds $X$ whose smooth toroidal compactifications $\overline{X}$ have non-ample canonical bundles. 
Moreover, each $\overline{X}$ contains a smooth rational curve $C$ not contained in the boundary $D=\overline{X}\setminus X$ and satisfying $K_{\overline{X}}\cdot C=0$.
\end{maintheorem}

Together, Theorem~\ref{thm:ample canonical} and Theorem~\ref{thm:threefolds} answer the question of Di Cerbo and Di Cerbo \cite[Question~1.8]{DiCerboDiCerbo} affirmatively and resolve the question of Bakker and Tsimerman \cite[Question after Corollary~D]{BakkerTsimerman} on ampleness of the canonical bundle completely.

More generally, our main result constructs rational curves meeting the interior in every dimension $n\geq2$:

\begin{maintheorem}\label{thm:all-dimensions}
For every integer $n\geq2$, there exist infinitely many pairwise noncommensurable noncompact Picard modular varieties $X$ of dimension $n$ whose smooth toroidal compactifications $\overline{X}$ contain rational curves $C$ not contained in the boundary $D=\overline{X}\setminus X$.
Moreover, each $C$ can be chosen smooth so that it meets $D$ transversely in four points on four distinct boundary components, $C\cap X$ is totally geodesic, and $K_{\overline{X}}\cdot C=n-3$.
\end{maintheorem}

For a complex projective variety $Y$, its \emph{algebraic exceptional locus} $\operatorname{Exc}_{\mathrm{alg}}(Y)$ and \emph{analytic exceptional locus} $\operatorname{Exc}_{\mathrm{ent}}(Y)$ are the Zariski closures of the unions of, respectively, its positive-dimensional subvarieties not of general type and the images of nonconstant holomorphic maps $\C\to Y$. For $Y$ of general type, the Lang conjecture and the Green--Griffiths--Lang conjecture predict that the respective loci are proper.

 The ampleness result of Bakker and Tsimerman follows from the estimate that $K_{\overline{X}}+D-rD$ is ample for $0<r<(n+1)d/(4\pi)$, where $d$ is the uniform cusp depth of $X$, together with the bound $d\geq2$ \cite[Proposition~3.6 and Corollaries~3.8--3.9]{BakkerTsimerman}.

If $d>4\pi$, the analytic exceptional locus $\operatorname{Exc}_{\mathrm{ent}}(\overline{X})$ is contained in $D$ by the preceding ampleness estimate and Nadel's theorem \cite[Theorem~2.1]{Nadel}. Under the same depth condition, the algebraic exceptional locus $\operatorname{Exc}_{\mathrm{alg}}(\overline{X})$ is also contained in $D$. This follows from the first author's ampleness theorem for $\Omega^1_{\overline{X}}(\log D)\langle-rD\rangle$, for rational $0<r<d/(4\pi)$, which implies that every positive-dimensional subvariety meeting $X$ is of general type \cite[Theorem~A and Corollary~F]{MemarianCotangent}.

The rational curves of Theorem~\ref{thm:all-dimensions} lie in both exceptional loci and are not contained in $D$. Therefore a depth condition cannot be omitted in any dimension. Although the range of ample twists in Bakker--Tsimerman's estimate grows with $n$, increasing the dimension alone does not confine either locus to $D$, even when $K_{\overline{X}}$ is ample. Cadorel \cite[Corollary~1.2]{Cadorel} proved that subvarieties of dimension at least $6$ in $\overline{X}$, not contained in $D$, are of general type. The rational curves in our examples show that the same conclusion does not hold for curves in $\overline{X}$.

\subsection*{Strategy of proof}
\phantomsection
\addcontentsline{toc}{subsection}{Strategy of proof}
To prove our main result, Theorem~\ref{thm:all-dimensions},
we use level-$3$ principal congruence subgroups of unitary groups
over a family of imaginary quadratic fields. The Hermitian form in each dimension contains a hyperbolic plane whose integral special unitary group is $\SL_2(\Z)$. An involution fixing this plane gives an embedded totally geodesic modular curve, which we identify with the modular curve $Y(3)$, the quotient of the upper half-plane by the principal congruence subgroup of level $3$ in $\SL_2(\Z)$. Its compactification is $\PP^1$ with four cusps. At each cusp, the curve and ambient quotient have the same shortest positive vertical translation, so the curve extends transversely to the toroidal boundary with multiplicity one. Reduction modulo $3$ distinguishes its four boundary components.

The Euler characteristic $\chi(Y(3))=-2$ determines the logarithmic canonical degree:
\[
 (K_{\overline{X}}+D)\cdot C
 =-\frac{n+1}{2}\chi(Y(3))=n+1.
\]
Subtracting $D\cdot C=4$ gives $K_{\overline{X}}\cdot C=n-3$. The remaining issue is to distinguish the quotients as the field varies. For $n\geq2$, the rational matrix algebra spanned by any finite-index subgroup of the lattice is the full matrix algebra over its defining field. The centre of this algebra recovers that field and is preserved by commensurability.

\subsection*{Organization of the paper}
\phantomsection
\addcontentsline{toc}{subsection}{Organization of the paper}

In Section~\ref{sec:Ampleness}, we prove
Theorem~\ref{thm:ample canonical}.
In Sections~\ref{sec:arithmetic}--\ref{sec:degrees}, we establish
the geometric properties of the Picard modular varieties used
in our construction.
In Section~\ref{sec:commensurability}, we prove pairwise
noncommensurability and complete the proof of
Theorem~\ref{thm:all-dimensions}.
Taking $n=3$, we obtain Theorem~\ref{thm:threefolds}, since
an ample divisor has positive degree on every projective curve.

\paragraph{Acknowledgements.}
S.M. carried out portions of this work while attending the International Congress of Mathematicians 2026, supported by an American Mathematical Society travel grant under NSF Grant No.~2529637. This material is based upon work supported by the National Science Foundation under Grant No.~2529637.

\paragraph{AI disclosure.}
GPT-5.6 Sol was used to produce
Figure~\ref{fig:cusp-identification-pictorial} and to assist
with language editing and polishing the text.

\section*{Notation and conventions}
\phantomsection
\addcontentsline{toc}{section}{Notation and conventions}

For a field $F$, we denote by $\Mat_m(F)$ the $F$-vector space of $m\times m$ matrices with entries in $F$. It is also an associative $F$-algebra under matrix multiplication. We write $\GL_m(F)$ for its group of invertible matrices and $\SL_m(F)$ for the subgroup of determinant one. The notation $\Span_F(S)$ means the vector space of finite $F$-linear combinations of elements of a set $S$.

The vectors $e_1,\ldots,e_m$ denote the standard basis of $F^m$. We use the same basis after extending scalars to $\C$. We write $I_m$ for the identity matrix and $E_{jk}$ for the matrix with entry $1$ in position $(j,k)$ and zero elsewhere. For a complex matrix $A$, its Hermitian transpose is $A^*={}^t\overline A$.

For a nondegenerate Hermitian matrix $G\in\Mat_m(\C)$, the corresponding Hermitian form is $h_G(v,w)=w^*Gv$. Its unitary, special unitary, and projective unitary groups are
\[
\begin{aligned}
\U(G)&=\{A\in\GL_m(\C):A^*GA=G\},\\
\SU(G)&=\{A\in\U(G):\det A=1\},\\
\PU(G)&=\U(G)/\{\lambda I_m:|\lambda|=1\}.
\end{aligned}
\]
If $K$ is an imaginary quadratic field with ring of integers $\OO_K$, then $\GL_m(\OO_K)$ consists of the matrices whose entries and inverse entries lie in $\OO_K$, and $\SL_m(\OO_K)$ is its determinant-one subgroup. We write $\U(G,\OO_K)=\U(G)\cap\GL_m(\OO_K)$ and $\SU(G,\OO_K)=\SU(G)\cap\SL_m(\OO_K)$. For $G=\diag(1,-1,\ldots,-1)$ of size $n+1$, the corresponding groups are denoted by $\U(1,n)$, $\SU(1,n)$, and $\PU(1,n)$. The group $\PU(1,n)$ acts faithfully on the positive complex lines and is the holomorphic isometry group of the unit ball $\B^n=\{z\in\C^n:\sum_{j=1}^n|z_j|^2<1\}.$ The map $\SU(1,n)\to\PU(1,n)$ has kernel $\{\lambda I_{n+1}:\lambda^{n+1}=1\}$. These conventions agree with the unitary-group description in \cite[\S 2.2]{MemarianVolumes}, with the positive-line convention for signature $(1,n)$ used here. We write $\chi$ for the topological Euler characteristic. 

A matrix group is \emph{neat} if, for every element, the multiplicative subgroup of $\C^\times$ generated by its eigenvalues contains no nonidentity root of unity. In particular, a neat group is torsion-free. The fields, lattices, and cusp coordinates used in the construction are introduced in the relevant sections.

\section{Ampleness of the canonical bundle}\label{sec:Ampleness}

For a smooth projective variety $Y$, we write $N^1(Y)_{\mathbb R}$
for its real N\'eron--Severi space, $N_1(Y)_{\mathbb R}$ for the
dual space of numerical $1$-cycle classes, and
$\overline{\operatorname{NE}}(Y)$ for the closed cone generated
by effective curves; see \cite[\S\S1.1.C and 1.4.C]{Lazarsfeld}.

\begin{lemma}\label{lem:canonical-rational-curve}
Let $Y$ be a smooth projective variety over $\mathbb C$ with $K_Y$
nef and big. If $K_Y$ is not ample, then $Y$ contains an integral
rational curve $C$ with $K_Y\cdot C=0$.
\end{lemma}
\begin{proof}
    Since $K_Y$ is nef but not ample, Kleiman's ampleness criterion
\cite[Theorem~1.4.29]{Lazarsfeld} gives a nonzero class
$\alpha\in\overline{\operatorname{NE}}(Y)$ such that
\(
K_Y\cdot\alpha=0.
\)
Since $K_Y$ is big, Kodaira's lemma
\cite[Corollary~2.2.7]{Lazarsfeld} gives
\(
K_Y\sim_{\mathbb Q}A+E,
\)
where $A$ is an ample $\mathbb Q$-divisor and $E$ is effective. Choose a sufficiently small rational $\varepsilon>0$ such that
$(Y,\varepsilon E)$ is klt. To see this, take a
log resolution $\pi\colon\widetilde Y\to Y$ and write
\[
K_{\widetilde Y}=\pi^*K_Y+\sum_i a_iF_i,
\qquad
\pi^*E=\sum_i b_iF_i,
\]
where the $F_i$ include the exceptional divisors and the strict
transforms of the components of $E$.
Smoothness of $Y$ gives $a_i\geq0$, and effectiveness of $E$
gives $b_i\geq0$. Choose $\varepsilon>0$ rational so that
$\varepsilon b_i<a_i+1$ for every $i$.
Such a choice is possible because there are only finitely many
inequalities, each with a positive upper bound when $b_i>0$.
Then
\[
K_{\widetilde Y}
=\pi^*(K_Y+\varepsilon E)
+\sum_i(a_i-\varepsilon b_i)F_i,
\]
and every discrepancy is greater than $-1$.
Since the total transform has simple normal crossings support,
this verifies that $(Y,\varepsilon E)$ is klt; see \cite[Definition~2.34]{KollarMori}.

Now set
\(M=K_Y+\varepsilon E+\frac{\varepsilon}{2}A.\)
The equality $K_Y\cdot\alpha=0$ and the decomposition above give
$E\cdot\alpha=-A\cdot\alpha$, so
\[
M\cdot\alpha
=\varepsilon E\cdot\alpha
+\frac{\varepsilon}{2}A\cdot\alpha
=-\frac{\varepsilon}{2}A\cdot\alpha<0,
\]
where the strict inequality follows from ampleness of $A$
and Kleiman's criterion. The cone theorem \cite[Theorem~3.7]{KollarMori}, applied to
$(Y,\varepsilon E)$ with the ample divisor
$\frac{\varepsilon}{2}A$, gives
\[
\alpha=\beta+\sum_{i=1}^{r}a_i[C_i],
\qquad
\beta\in\overline{\operatorname{NE}}(Y),\quad
M\cdot\beta\geq0,\quad a_i\geq0,
\]
where the $C_i$ are integral rational curves.
At least one $a_i$ is positive because $M\cdot\alpha<0$.
As $K_Y$ is nef,
\[
0=K_Y\cdot\alpha
=K_Y\cdot\beta+\sum_{i=1}^{r}a_i(K_Y\cdot C_i)
\]
is a sum of nonnegative terms. Hence $K_Y\cdot C_i=0$
whenever $a_i>0$.
\end{proof}

Now we can prove Theorem \ref{thm:ample canonical}: 
\begin{proof}[Proof of Theorem \ref{thm:ample canonical}]
The canonical divisor $K_{\overline X}$ is nef by \cite[Theorem~1.1]{DiCerboDiCerbo}, and big by
\cite[Theorem~C]{BakkerTsimerman}. If it were not ample,
Lemma~\ref{lem:canonical-rational-curve} would give an integral
rational curve $C\subset\overline X$ with
$K_{\overline X}\cdot C=0$. This contradicts
\cite[Lemma~4.7]{BakkerTsimerman}, since $n\geq4$.
\end{proof}
\section{The arithmetic lattices and the modular curve}\label{sec:arithmetic}

In this section we choose the imaginary quadratic fields and construct a totally geodesic modular curve in the associated Picard modular varieties. The choice of Hermitian form keeps the integral group acting on the curve equal to $\SL_2(\Z)$ as the field varies.

Fix a dimension $n\geq2$. Let $L$ be an imaginary quadratic field with ring of integers $\OO_L$, and let $G$ be a Hermitian form on $L^{n+1}$ of signature $(1,n)$. The space of positive complex lines for $G$ is biholomorphic to the unit ball $\B^n$. A quotient of this ball by a lattice commensurable, after projectivization and conjugation, with $\SU(G,\OO_L)$ is called a \emph{Picard modular variety}. We choose the defining fields as follows.

Fix a prime $p\equiv1\pmod{12}$. Let $K=\Q(i\sqrt p)$. Since $p\equiv1\pmod4$, its ring of integers is $\OO_K=\Z[i\sqrt p]$ and the discriminant of $K$ is $-4p.$ Consider the Hermitian form on $K^{n+1}$ given by
\begin{equation}\label{eq:hermitian-form}
J=\begin{pmatrix}0&i\sqrt p\\-i\sqrt p&0\end{pmatrix},
\qquad H=J\oplus(-I_{n-1}).
\end{equation}
Since $J$ has eigenvalues $\pm\sqrt p$, $H$ has signature $(1,n)$.
It is isotropic over $K$ because $e_1$ is isotropic.

Let $\B_H$ be the space of positive complex lines $[v]\in\PP^n(\C)$ satisfying $v^*Hv>0$. A complex linear change of coordinates puts $H$ in the form $|v_0|^2-\sum_{j=1}^n|v_j|^2$. The ratios $v_j/v_0$ then identify $\B_H$ with the unit ball $\B^n$. Define
\begin{equation}\label{eq:lattice}
\Gamma=\{\gamma\in\SL_{n+1}(\OO_K):
\gamma^*H\gamma=H,\ \gamma\equiv I_{n+1}\pmod{3\OO_K}\}.
\end{equation}
By the arithmetic lattice theorem \cite{BorelHarishChandra}, $\SU(H,\OO_K)$ is a lattice in $\SU(H)\simeq\SU(1,n)$. Its isotropy over $K$ makes it nonuniform. The finite-index subgroup $\Gamma$ is neat by \cite[Lemma~1.4]{DimitrovRamakrishnan}; see also the original
\cite[Lemma~4.3]{Hol80}. Its projective action is faithful and free, and its cusp stabilizers are unipotent. Hence
\[
X=\Gamma\backslash\B_H
\]
is a smooth noncompact Picard modular variety with a smooth projective toroidal compactification $\overline X$ and boundary $D=\overline X\setminus X$.

Let $W=\Span_\C\{e_1,e_2\}$. The restriction of $H$ to $W$ is represented by $J$ and has signature $(1,1)$, so the positive lines in $W$ form a copy $\B_W\simeq\HH$ of the upper half-plane. Consider the integral special unitary involution
\begin{equation}\label{eq:involution}
 r=\diag(-I_2,I_{n-1})\in\SU(H,\OO_K).
\end{equation}
 Conjugation by $r$ preserves the congruence condition defining $\Gamma$, which gives $r\Gamma r^{-1}=\Gamma$. The fixed-point locus of $r$ in $\B_H$ is exactly $\B_W$. Since $r$ acts as an isometry, this fixed-point locus is totally geodesic.

\begin{lemma}\label{lem:stabilizer}
The stabilizer $\Gamma_W=\{\gamma\in\Gamma:\gamma W=W\}$ acts faithfully on $\B_W$, and its quotient is the modular curve $Y(3)=\Gmod\backslash\HH$, where
\[
\Gmod=\ker\bigl(\SL_2(\Z)\longrightarrow\SL_2(\Z/3\Z)\bigr).
\]
\end{lemma}

\begin{proof}
An element of $\Gamma_W$ preserves $W\oplus W^\perp$, where orthogonality is with respect to $H$, and has the form $\diag(A,B)$ with $A\in\U(J,\OO_K)$ and $B\in\U(I_{n-1},\OO_K)$. The latter integral unitary group is finite because it is the intersection of a compact matrix group with the discrete set of matrices over $\OO_K$. Hence every eigenvalue of $B$ is a root of unity. These are also eigenvalues of $\diag(A,B)\in\Gamma$, so neatness forces them all to be $1$. A finite-order unipotent matrix is the identity. Therefore $B=I_{n-1}$, and the determinant-one condition gives $\det A=1$.

For this choice of $J$, the identities $A^tJA=(\det A)J$ and $A^*JA=J$ force every element of $\SU(J,\OO_K)$ to have real entries. Since $\OO_K\cap\R=\Z$ and every matrix in $\SL_2(\Z)$ preserves $J$, we have $\SU(J,\OO_K)=\SL_2(\Z)$. Moreover, $3\OO_K\cap\Z=3\Z$, so this identification preserves the principal congruence subgroups of level $3$. Block extension therefore gives the isomorphism
\begin{equation}\label{eq:sl2-identification}
\Phi:\Gmod\xrightarrow{\ \sim\ }\Gamma_W,\qquad
M\longmapsto\diag(M,I_{n-1}).
\end{equation}

Under the identification $\HH\simeq\B_W$ given by $\tau\mapsto[\tau:1:0]$, this isomorphism identifies the two actions by M\"obius transformations. Therefore
\[
\Gamma_W\backslash\B_W\simeq\Gmod\backslash\HH=Y(3).
\]
Finally, an element acting trivially on $\B_W$ must have $A=\pm I_2$. The congruence condition excludes $-I_2$, so the action is faithful.
\end{proof}

\begin{lemma}\label{lem:injectivity}
The natural totally geodesic immersion $Y(3)=\Gamma_W\backslash\B_W\to X$ is injective.
\end{lemma}

\begin{proof}
We must show that any element of $\Gamma$ identifying two points of $\B_W$ already belongs to $\Gamma_W$. Suppose $x,y\in\B_W$ and $y=\gamma x$ for $\gamma\in\Gamma$. Since $r$ fixes both points, the element $\delta=\gamma^{-1}r\gamma r^{-1}$ fixes $x$. It belongs to $\Gamma$, because $r$ normalizes $\Gamma$. The stabilizer of $x$ in a discrete group of ball isometries is finite, while $\Gamma$ is torsion-free and acts faithfully. Hence $\delta=1$, so $\gamma$ commutes with $r$ and preserves $W$. Therefore $\gamma\in\Gamma_W$.
\end{proof}

\section{Extension to the Compactification}\label{sec:compact-curve}

We now extend the embedded modular curve to $\overline X$. To understand the boundary behavior, we identify the cusps with isotropic lines and compare the toroidal coordinates of the curve and ambient quotient. The cusp-neighborhood construction follows Parker \cite{Parker} and Bakker--Tsimerman \cite[\S 2]{BakkerTsimerman}, with the horospherical notation of \cite[\S\S 2.2, 2.5--2.6]{MemarianVolumes}.

Let $X(3)$ denote the smooth projective compactification of $Y(3)$. Its added points, called cusps, correspond to $\Gmod$-orbits in $\PP^1(\Q)=\Q\cup\{\infty\}$, the rational boundary points of $\HH$. The curve $X(3)$ has genus zero and four cusps. Write the added points as $s_1,\ldots,s_4$. In particular, $X(3)\simeq\PP^1$ and
\begin{equation}\label{eq:modular-euler}
\chi(Y(3))=2-2g(X(3))-4=-2.
\end{equation}

\paragraph{Cusps and isotropic lines.}
The ideal boundary of $\B_H$ consists of complex lines $\C a$ with $a\neq0$ and $a^*Ha=0$. Such a line is called \emph{isotropic}. It is \emph{defined over $K$} if it has a generator $a\in K^{n+1}$. For the arithmetic lattice $\Gamma$, its parabolic fixed points are exactly these $K$-defined isotropic lines, and a cusp of $X$ is a $\Gamma$-orbit of such lines; see Looijenga \cite[\S\S 4.1--4.3]{LooijengaBall}. The stabilizer of such a line in the special unitary group is a rational parabolic subgroup. This is the ball case of the Baily--Borel correspondence between rational boundary components and rational parabolic subgroups \cite[Theorem~3.7]{BailyBorel}; the boundary of the quotient is indexed by their $\Gamma$-orbits \cite[Corollary~4.11]{BailyBorel}.

The Baily--Borel compactification adds one point for each cusp, and the toroidal compactification replaces that point by a boundary component of $D.$ Two cusp lines therefore give the same boundary component precisely when they belong to the same $\Gamma$-orbit. For the modular curve, cusps are $\Gamma_W$-orbits of the $K$-defined isotropic lines contained in $W$. Under the coordinate $\tau$, a finite rational boundary point $\tau=a/c$, with $a,c\in\Z$ coprime and $c\neq0$, corresponds to $\C(ae_1+ce_2)$, and $\infty$ corresponds to $\C e_1$. These are all the $K$-defined isotropic lines in $W$, because $K\cap\R=\Q$. The inclusion of the curve sends each such line to the same isotropic line in the ambient ball. This identifies the points $s_j$ with the ambient cusps used below.

\paragraph{Toroidal coordinates at a cusp.}
In the affine chart of positive lines $[\tau:1:\xi]$, with $\tau\in\C$ and $\xi\in\C^{n-1}$, the ball is given by
\(2\sqrt p\operatorname{Im}\tau-\|\xi\|^2>0.\)
Set $\zeta=\xi$ and $z=2\sqrt p\,\tau$. These identify $\B_H$ with the Siegel domain
\[
S=\{(\zeta,z)\in\C^{n-1}\times\C:
\operatorname{Im}z>\|\zeta\|^2\}.
\]
Its horospherical coordinates are $(\zeta,v,u)$, where
\[
v=\operatorname{Re}z,\qquad
u=\operatorname{Im}z-\|\zeta\|^2>0,
\qquad z=v+i(\|\zeta\|^2+u).
\]
The distinguished boundary point $q_\infty$ corresponds to $\C e_1$, and the modular geodesic $\B_W$ is the locus $\zeta=0$. Let $\Gamma_\infty$ be the stabilizer of $q_\infty$ in $\Gamma$. Its elements are unipotent and act by Heisenberg translations
\[
(\eta,t):(\zeta,v,u)\longmapsto
(\zeta+\eta,v+t+2\operatorname{Im}\langle\eta,\zeta\rangle,u),
\]
where $\eta\in\C^{n-1}$, $t\in\R$, and $\langle\eta,\zeta\rangle=\sum_j\eta_j\overline\zeta_j$. The centre of the Heisenberg group is the subgroup $V_\infty$ of \emph{vertical translations} $(0,t)$. Write
\(
\Gamma_\infty\cap V_\infty
=\{(0,kt_\infty):k\in\Z\},
\)
where $t_\infty>0$ is the parameter of the \emph{shortest positive vertical translation}.

For $\tilde u>0$, let $B_\infty(\tilde u)=\{(\zeta,v,u)\in S:u>\tilde u\}$. For sufficiently large $\tilde u$, the quotient $\Gamma_\infty\backslash B_\infty(\tilde u)$ is an embedded cusp neighborhood in $X$; see \cite{Parker} and \cite[\S 2]{BakkerTsimerman}. To compactify this neighborhood, first divide by $\Gamma_\infty\cap V_\infty$. The quotient is described by
\begin{equation}\label{eq:toroidal-coordinate}
(\zeta,z)\longmapsto(\zeta,q),\qquad
q=\exp\!\left(\frac{2\pi iz}{t_\infty}\right),
\end{equation}
with image
\(
\left\{(\zeta,q):
0<|q|<\exp\!\left(-\frac{2\pi}{t_\infty}
(\|\zeta\|^2+\tilde u)\right)\right\}.
\)
Adding $q=0$ and then taking the residual quotient by $\Gamma_\infty/(\Gamma_\infty\cap V_\infty)$ gives the toroidal cusp neighborhood. Because $\Gamma$ is neat, the residual quotient is locally biholomorphic along the added boundary. Hence, after choosing a local lift, $(\zeta,q)$ are holomorphic coordinates in which $D$ is defined by $q=0$.

\begin{lemma}\label{lem:cusp-width}
The immersion $Y(3)\to X$ extends to a morphism $f:X(3)\to\overline{X}$ satisfying
\[
f^*D=s_1+\cdots+s_4.
\]
Locally at every $s_j$, the map is an immersion transverse to $D$.
\end{lemma}

\begin{proof}
The boundary of the toroidal compactification of $X$ is locally defined by $q=0$. We will show that
$q$ restricts to the coordinate used to fill the curve cusp.
This will extend the map and show that the boundary equation
has a simple zero on the completed curve.

At $q_\infty$, the curve is given by $\zeta=0$ and
$z=2\sqrt p\,\tau$. By Lemma~\ref{lem:stabilizer}, its cusp
stabilizer acts by $\tau\mapsto\tau+3k$, with $k\in\Z$.
The quotient near the cusp is therefore a punctured disk
with coordinate
\(
\widetilde q=\exp\!\left(\frac{2\pi i\tau}{3}\right),
\)
and adding the cusp to $X(3)$ means filling this disk at
$\widetilde q=0$.

We now express the ambient boundary coordinate $q$ in terms
of $\widetilde q$. A vertical translation $z\mapsto z+t$
is represented by $I_{n+1}+\frac{t}{2\sqrt p}E_{12}$.
Since $\OO_K\cap\R=\Z$, this matrix belongs to $\Gamma$
precisely when $t/(2\sqrt p)\in3\Z$.
The shortest positive ambient period is therefore
$t_\infty=6\sqrt p$, and on the curve we have
\[
q=\exp\!\left(\frac{2\pi iz}{6\sqrt p}\right)
=\exp\!\left(\frac{2\pi i\tau}{3}\right)=\widetilde q.
\]
Consequently, in the ambient coordinates $(\zeta,q)$,
the inclusion is
$\widetilde q\mapsto(0,\widetilde q)$.
It extends across $\widetilde q=0$ by sending the added
cusp to $(0,0)$. Its differential is $(0,1)$, and the
boundary equation pulls back to $\widetilde q$.
Hence the extension is an immersion transverse to $D$,
with intersection multiplicity one.

For another cusp, choose $g\in\SL_2(\Z)$ carrying $\C e_1$
to its rational isotropic line. The block extension
$\diag(g,I_{n-1})$ preserves $H$ and normalizes $\Gamma$,
so it transports the same calculation to that cusp.
The four local extensions agree with the map on $Y(3)$
and give a holomorphic map $f:X(3)\to\overline X$.
Since both varieties are projective, $f$ is algebraic.
Each added cusp meets $D$ with multiplicity one, while
$f(Y(3))\subset X$, so $f^*D=s_1+\cdots+s_4$.
\end{proof}

\begin{remark}[Necessity of distinguishing cusps]
Let $C'$ be a quasi-projective curve with distinct cusps $s_1$ and $s_2$,
and let
\(
h\colon C'\longrightarrow X
\)
be an injective morphism. Even if $U_1$ and $U_2$ are disjoint cusp
neighborhoods of $s_1$ and $s_2$, injectivity of $h$ does not imply that
they approach distinct Baily--Borel cusps of $X$. In principle, both
$h(U_1)$ and $h(U_2)$ could approach the same cusp $p_\infty$, as
illustrated in Figure~\ref{fig:cusp-identification-pictorial}.
Proposition~\ref{prop:compact-embedding} excludes this possibility in our
construction by distinguishing the $\Gamma$-orbits of the four cusp lines.
\end{remark}

\begin{figure}[htbp]
  \centering
  \includegraphics[width=.44\textwidth]
  {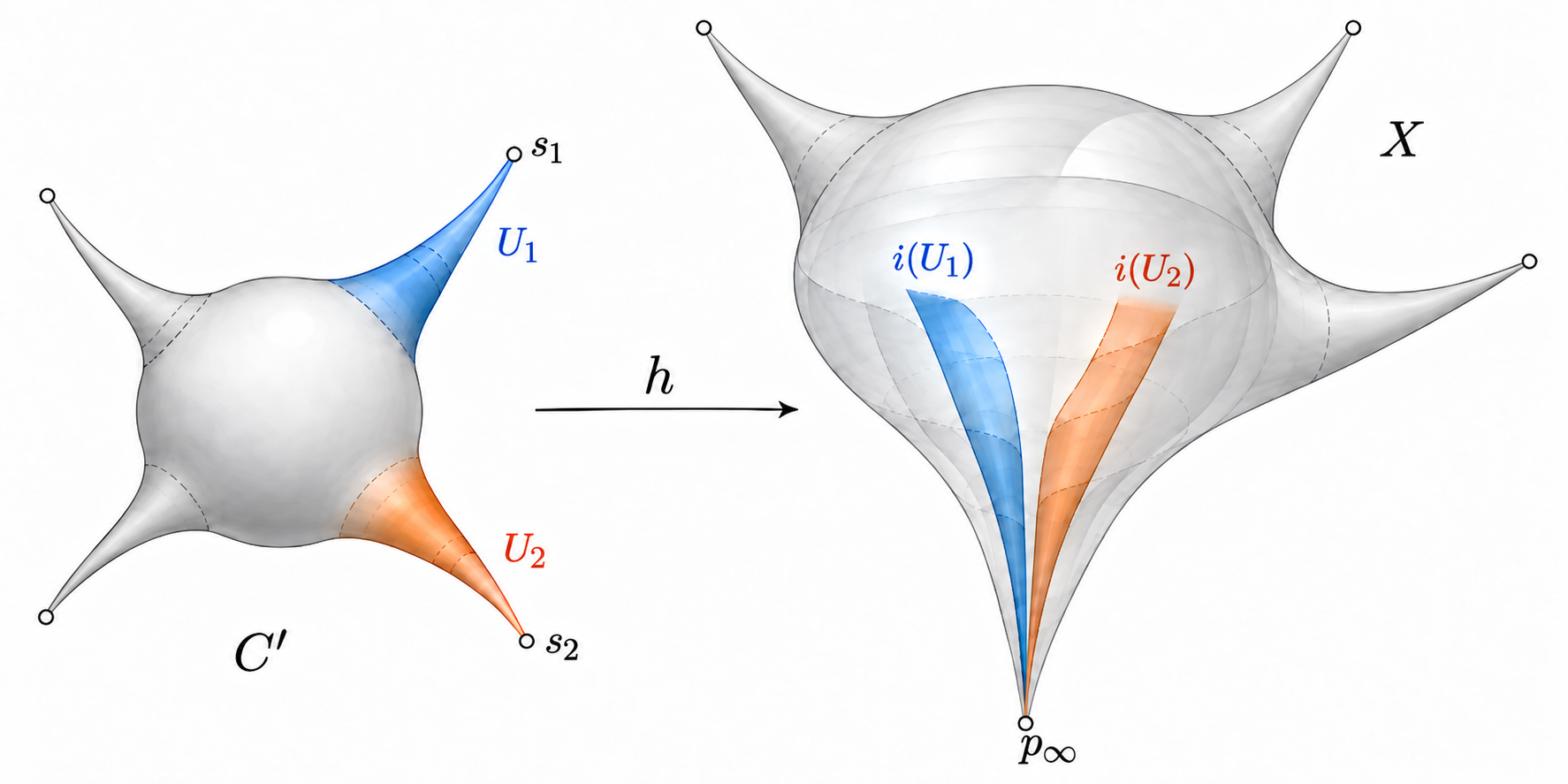}
  \caption{A schematic configuration in which two disjoint cusp neighborhoods $U_1$ and $U_2$ of a curve $C'$ may approach the same Baily--Borel cusp $p_\infty$ of $X$ under an injective map $h$. Proposition~\ref{prop:compact-embedding} excludes this configuration in our construction.}
  \label{fig:cusp-identification-pictorial}
\end{figure}

\begin{proposition}\label{prop:compact-embedding}
The map $f$ of Lemma~\ref{lem:cusp-width} is a closed embedding. Its image $C\simeq\PP^1$ meets $D$ transversely in four distinct points on four distinct boundary components, and $C\cap X$ is the embedded totally geodesic modular curve $Y(3)$.
\end{proposition}

\begin{proof}
We first show that the four cusps of $Y(3)$ map to distinct boundary components. Consider the primitive isotropic vectors
\[
v_1=e_1,\qquad v_2=e_2,\qquad
v_3=e_1+e_2,\qquad v_4=e_1-e_2.
\]
Their lines are images of $\C e_1$ under $\SL_2(\Z)$. Since $p\equiv1\pmod3$, we have
\begin{equation}\label{eq:inert-three}
\OO_K/3\OO_K\simeq\mathbb F_3[T]/(T^2+1)\simeq\mathbb F_9,
\end{equation}
and their reductions are the distinct points
$[1:0]$, $[0:1]$, $[1:1]$, and $[1:-1]$ in
$\PP^1(\mathbb F_9)$.

If $\gamma(\C v_j)=\C v_k$ for some $\gamma\in\Gamma$, then
$\gamma v_j=\lambda v_k$ with $\lambda\in K^\times$.
Primitivity and the integrality of $\gamma$ and $\gamma^{-1}$
give $\lambda\in\OO_K^\times$. Reducing modulo $3$ and using
$\gamma\equiv I_{n+1}$ therefore gives that 
$v_j\equiv\lambda v_k\pmod{3\OO_K},$ so $j=k$.
The four lines belong to distinct $\Gamma$-orbits, hence also
to distinct $\Gamma_W$-orbits. Since $Y(3)$ has four cusps,
they represent all of them, and their images lie on four
distinct boundary components.

It remains to prove that $f$ is a closed embedding.
By Lemma~\ref{lem:injectivity}, $f$ is injective on $Y(3)$.
The four added points have distinct images in $D$, while
$f(Y(3))\subset X$, so $f$ is injective on all of $X(3)$.
It is an immersion on the open curve by construction and
at every cusp by Lemma~\ref{lem:cusp-width}.
Since $X(3)$ is compact, this injective holomorphic immersion
is a closed embedding. The same lemma gives transversality
to $D$, and the description of $C\cap X$ follows from
Section~\ref{sec:arithmetic}.

\end{proof}

\section{The canonical degree}\label{sec:degrees}

Normalize the complex hyperbolic K\"ahler form $\omega$ on $\B^n$ so that its metric has constant holomorphic sectional curvature $-1$, and use the same symbol for the descended form on $X$. Its Ricci form is $\Ric(\omega)=-(n+1)\omega/2$, so the logarithmic Chern--Weil identity is
\begin{equation}\label{eq:chern-weil}
c_1(K_{\overline{X}}+D)=\frac{n+1}{4\pi}[\omega].
\end{equation}
Here $[\omega]$ is the class of the extended current. The identity, including its boundary extension, follows from Mumford's theory of good metrics \cite{Mumford}.

\begin{proposition}\label{prop:degrees}
For the curve $C$ of Proposition~\ref{prop:compact-embedding}, one has
\[
(K_{\overline{X}}+D)\cdot C=n+1,\qquad
D\cdot C=4,\qquad K_{\overline{X}}\cdot C=n-3.
\]
\end{proposition}

\begin{proof}
Since the inclusion is totally geodesic, the induced metric on $C\cap X\simeq Y(3)$ has Gaussian curvature $-1$. The invariant metric on $K_X$ is good in the sense of Mumford with respect to the extension $K_{\overline X}+D$, and its Chern form has Poincaré growth along $D$. Since $C$ meets $D$ transversely, its restriction has Poincaré growth at the four cusps of $C\cap X$. Applying \eqref{eq:chern-weil} and Gauss--Bonnet, and using \eqref{eq:modular-euler}, gives
\[
(K_{\overline{X}}+D)\cdot C
=-\frac{n+1}{2}\chi(C\cap X)
=-\frac{n+1}{2}\chi(Y(3))=n+1.
\]
Subtracting $D\cdot C=4$, established in Proposition~\ref{prop:compact-embedding}, yields
\(
K_{\overline{X}}\cdot C=n-3.
\)
\end{proof}

The geometry of Theorem~\ref{thm:all-dimensions} is now established for every prime $p\equiv1\pmod{12}$ and $n\geq2$. Only the distinction between commensurability classes remains.

\section{Noncommensurability and the main result}\label{sec:commensurability}

Fix a dimension $n\geq2$, and put $m=n+1$ and  
$\Pset=\{p\text{ prime}:p\equiv1\pmod{12}\}$. We now allow the prime $p$ to vary in $\Pset.$
In this section, we write $\Gamma_p$, $X_p$, $\overline X_p$, $D_p$, and $C_p$ for the lattice, variety, compactification, boundary, and curve of Sections~\ref{sec:arithmetic}--\ref{sec:degrees}. 

Let $\Gamma_p^{\mathrm{proj}}\subset\PU(H_p)$ denote the projective image of $\Gamma_p$. Two quotients are \emph{commensurable} if, after identifying their ball models with $\B^n$, their projective lattices are conjugate to subgroups whose intersection has finite index in both.

Our goal is to recover $K_p$ from any finite-index subgroup of $\Gamma_p$. We use its rational linear span inside $\Mat_m(\C)$. Because the subgroup contains the identity and is closed under multiplication, this span is a unital $\Q$-algebra. The next lemma identifies it with $\Mat_m(K_p)$, whose centre is the scalar copy of $K_p$. We will then show that projective commensurability gives conjugate rational algebras and hence equal fields.

\begin{lemma}\label{lem:span}
For every subgroup $\Gamma'_p$ of finite index in $\Gamma_p$,
\[
\Span_\Q(\Gamma'_p)=\Mat_m(K_p).
\]
\end{lemma}

\begin{proof}
To prove the rational-span statement, we first determine the real span using Borel density. We then use rational coordinates for the entries in $K_p$ to descend the dimension calculation to $\Q$.

Set $\mathcal A=\Span_\R(\Gamma'_p)\subset\Mat_m(\C)$. This is a unital real algebra. As a real linear subspace of a finite-dimensional matrix space, it is real Zariski closed. The subgroup $\Gamma'_p$ is a lattice in the connected real algebraic group $\SU(H_p)$, which has no compact factors. By Borel density \cite[Corollary~4.5.6]{Morris}, its real Zariski closure is $\SU(H_p)$. Consequently $\SU(H_p)\subset\mathcal A$.

The complex Zariski closure of $\SU(H_p)$ is $\SL_m(\C)$, since $\SU(H_p)$ is a real form of that complex group. Hence its complex linear span contains $\SL_m(\C)$. This group spans $\Mat_m(\C)$: the differences $(I_m+E_{jk})-I_m$ give the off-diagonal matrix units, and the diagonal matrices of determinant one span the diagonal subspace.

We still need to see that the real algebra $\mathcal A$ contains this complex span. Because $m\geq3$, the scalar $\zeta=\exp(2\pi i/m)$ is nonreal, and $\zeta I_m\in\SU(H_p)\subset\mathcal A$. Therefore
\[
iI_m=\frac{\zeta I_m-(\operatorname{Re}\zeta)I_m}{\operatorname{Im}\zeta}
\in\mathcal A.
\]
Closure under multiplication makes $\mathcal A$ a complex linear subspace. It contains the complex span of $\SU(H_p)$, so $\mathcal A=\Mat_m(\C)$.

Finally, write each entry of a matrix over $K_p$ uniquely as $a+ib\sqrt{p}$ with $a,b\in\Q$. This identifies $\Mat_m(K_p)$ with $\Q^{2m^2}$; the same coordinates identify $\Mat_m(\C)$, as a real vector space, with $\R^{2m^2}$. The coordinate vectors of elements of $\Gamma'_p$ are rational, so their rank over $\Q$ equals their rank over $\R$. The full real span just proved has dimension $2m^2$. Their rational span consequently has that same dimension and equals $\Mat_m(K_p)$.
\end{proof}

The centre of an algebra consists of the elements commuting with every element of that algebra. For the algebra in Lemma~\ref{lem:span}, it is
\(
\{aI_m:a\in K_p\}.
\)
To compare these centres under commensurability, we must account for the scalar ambiguity of a projective conjugating map. The scalar factors are roots of unity of bounded order; passing to a further finite-index subgroup removes them.

\begin{proposition}\label{prop:noncommensurable}
For every fixed $n\geq2$ and distinct $p,q\in\Pset$, the Picard modular varieties $X_p$ and $X_q$ are not commensurable.
\end{proposition}

\begin{proof}
We show that commensurability would force $K_p=K_q$. First we lift the projective conjugacy to an exact matrix conjugacy on finite-index subgroups. Lemma~\ref{lem:span} will then identify the conjugate algebras, and their centres will determine the fields:

Suppose the quotients are commensurable, and let $T\in\GL_m(\C)$ represent the resulting projective isometry from $\B_{H_p}$ to $\B_{H_q}$. Let $\Lambda$ be the finite-index intersection
\[
\Lambda=(T\Gamma_p^{\mathrm{proj}}T^{-1})\cap\Gamma_q^{\mathrm{proj}},
\]
where conjugation is interpreted projectively. The inverse images of $T^{-1}\Lambda T$ in $\Gamma_p$ and of $\Lambda$ in $\Gamma_q$ are finite-index subgroups $\Gamma'_p$ and $\Gamma'_q$. Projectivization is injective on both neat lattices, so this conjugacy induces an isomorphism $\varphi:\Gamma'_p\to\Gamma'_q$. For each $\gamma\in\Gamma'_p$, there is a unique scalar $\varepsilon_\gamma$ with
\[
T\gamma T^{-1}=\varepsilon_\gamma\varphi(\gamma).
\]
Both matrices have determinant one, hence $\varepsilon_\gamma\in\mu_m$, where $\mu_m=\{\varepsilon\in\C^\times:\varepsilon^m=1\}$ is the finite group of $m$th roots of unity. Since $\varphi$ is a homomorphism, $\gamma\mapsto\varepsilon_\gamma$ is a homomorphism to $\mu_m$. Its kernel $\Gamma''_p$ has finite index. Put $\Gamma''_q=\varphi(\Gamma''_p)$. We now have an exact equality of matrix groups,
\[
T\Gamma''_pT^{-1}=\Gamma''_q.
\]
Taking rational spans and applying Lemma~\ref{lem:span} gives
\(
T\Mat_m(K_p)T^{-1}=\Mat_m(K_q).
\)
Conjugation carries the centre of one algebra to the centre of the other. But it fixes every complex scalar matrix, so
\[
\{aI_m:a\in K_p\}=\{aI_m:a\in K_q\}.
\]
Therefore $K_p=K_q$ as subfields of $\C$. This contradicts their distinct discriminants $-4p$ and $-4q$. Hence $X_p$ and $X_q$ are not commensurable.
\end{proof}

\begin{proof}[Proof of Theorem~\ref{thm:all-dimensions}]
Fix $n\geq2$. For each $p\in\Pset$, Sections~\ref{sec:arithmetic} and \ref{sec:compact-curve} construct a smooth noncompact Picard modular variety $X_p$ whose smooth projective toroidal compactification contains a smooth rational curve $C_p$. By Proposition~\ref{prop:compact-embedding}, the curve meets the boundary transversely in four points on four distinct components, and its open part is the embedded totally geodesic curve $Y(3)$. Proposition~\ref{prop:degrees} gives $K_{\overline X_p}\cdot C_p=n-3$. The set $\Pset$ is infinite, and Proposition~\ref{prop:noncommensurable} distinguishes every pair of its quotients. These statements prove the theorem.
\end{proof}

\begin{proof}[Proof of Theorem~\ref{thm:threefolds}]
Take $n=3$ in Theorem~\ref{thm:all-dimensions}. Each constructed curve has $K_{\overline X}\cdot C=0$, whereas an ample divisor has positive degree on every projective curve. Hence $K_{\overline X}$ is not ample.
\end{proof}

\bigskip
\begingroup
\small
\setlength{\parskip}{6pt}
\noindent\address{Department of Mathematics, Purdue University, West Lafayette, IN 47907, USA};
\email{smemaria@purdue.edu}\par
\noindent\address{Department of Mathematics, Purdue University, West Lafayette, IN 47907, USA};
\email{yeungs@purdue.edu}\par
\endgroup


\begin{thebibliography}{DCDC17}
\phantomsection
\addcontentsline{toc}{section}{References}

\bibitem[BB66]{BailyBorel}
W.~L. Baily, Jr. and A.~Borel,
\emph{Compactification of arithmetic quotients of bounded symmetric domains},
Ann. of Math. (2) \textbf{84} (1966), no.~3, 442--528.

\bibitem[BT18]{BakkerTsimerman}
B.~Bakker and J.~Tsimerman,
\emph{The Kodaira dimension of complex hyperbolic manifolds with cusps},
Compos. Math. \textbf{154} (2018), no.~3, 549--564.

\bibitem[BHC62]{BorelHarishChandra}
A.~Borel and Harish-Chandra,
\emph{Arithmetic subgroups of algebraic groups},
Ann. of Math. (2) \textbf{75} (1962), 485--535.

\bibitem[Cad22]{Cadorel}
B.~Cadorel,
\emph{Subvarieties of quotients of bounded symmetric domains},
Math. Ann. \textbf{384} (2022), 1--38; arXiv:1809.10978.

\bibitem[DCDC17]{DiCerboDiCerbo}
G.~Di Cerbo and L.~F. Di Cerbo,
\emph{On the canonical divisor of smooth toroidal compactifications},
Math. Res. Lett. \textbf{24} (2017), no.~4, 1005--1022.

\bibitem[DS16]{DiCerboStoverMultiple}
L.~F. Di Cerbo and M.~Stover,
\emph{Multiple realizations of varieties as ball quotient compactifications},
Michigan Math. J. \textbf{65} (2016), no.~2, 441--447.

\bibitem[DS18]{DiCerboStover}
L.~F. Di Cerbo and M.~Stover,
\emph{Classification and arithmeticity of toroidal compactifications with $3\overline c_2=\overline c_1^2=3$},
Geom. Topol. \textbf{22} (2018), no.~4, 2465--2510.

\bibitem[DR15]{DimitrovRamakrishnan}
M.~Dimitrov and D.~Ramakrishnan,
\emph{Arithmetic quotients of the complex ball and a conjecture of Lang},
Doc. Math. \textbf{20} (2015), 1185--1205.

\bibitem[Hir84]{Hirzebruch}
F.~Hirzebruch,
\emph{Chern numbers of algebraic surfaces: an example},
Math. Ann. \textbf{266} (1984), 351--356.


\bibitem[Hol80]{Hol80}
R. P.~Holzapfel, \emph{A class of minimal surfaces in the unknown region of surface geography}.
Math. Nachr. \textbf{98 }(1980), 211-232.

 \bibitem[KM98]{KollarMori}
 J.~Koll\'ar and S.~Mori,
 \emph{Birational geometry of algebraic varieties},
 Cambridge Tracts in Mathematics, vol.~134,
 Cambridge University Press, Cambridge, 1998.

\bibitem[Laz04]{Lazarsfeld}
 R.~Lazarsfeld,
 \emph{Positivity in algebraic geometry I: Classical setting:
 line bundles and linear series},
 Ergebnisse der Mathematik und ihrer Grenzgebiete (3), vol.~48,
 Springer-Verlag, Berlin, 2004.
 
\bibitem[Loo03]{LooijengaBall}
E.~Looijenga,
\emph{Compactifications defined by arrangements, I: The ball quotient case},
Duke Math. J. \textbf{118} (2003), no.~1, 151--187.

 
\bibitem[Mem24]{MemarianCotangent}
S.~Memariansorkhabi,
\emph{Positivity of the cotangent bundle of complex hyperbolic manifolds with cusps},
arXiv:2212.10816, version~3 (2024).

\bibitem[Mem25]{MemarianVolumes}
S.~Memariansorkhabi,
\emph{Volumes of subvarieties of complex ball quotients and sparsity of rational points},
J. Reine Angew. Math. \textbf{829} (2025), 177--215.

\bibitem[Mor15]{Morris}
D.~Witte Morris,
\emph{Introduction to Arithmetic Groups},
arXiv:math/0106063, version~6 (2015).

\bibitem[Mum77]{Mumford}
D.~Mumford,
\emph{Hirzebruch's proportionality theorem in the non-compact case},
Invent. Math. \textbf{42} (1977), 239--272.

\bibitem[Nad89]{Nadel}
A.~M. Nadel,
\emph{The nonexistence of certain level structures on abelian varieties over complex function fields},
Ann. of Math. (2) \textbf{129} (1989), no.~1, 161--178.

\bibitem[Par98]{Parker}
J.~R. Parker,
\emph{On the volumes of cusped, complex hyperbolic manifolds and orbifolds},
Duke Math. J. \textbf{94} (1998), no.~3, 433--464.

\bibitem[PY07]{PrasadYeung}
G.~Prasad and S.-K.~Yeung,   
\emph{Fake projective planes}, Inv.\,Math. \textbf{168}(2007), 321-370; Addendum, ibid \textbf{182}(2010), 213-227.

\end{thebibliography}
\end{document}